\documentclass{lms}
\pdfoutput=1
\usepackage[utf8]{inputenc}
\usepackage[T1]{fontenc}
\usepackage[english]{babel}

\usepackage{graphicx}
\let\oldincgraphics\includegraphics
\renewcommand{\includegraphics}[1]{\oldincgraphics{img/#1}}

\usepackage{amsmath}
\usepackage{amsfonts}
\usepackage{amssymb}

\usepackage{url}
\usepackage{multicol}
\makeatletter
\renewcommand{\fps@figure}{htbp}
\renewcommand{\fps@table}{htbp}
\makeatother

\usepackage[vlined,algosection]{algorithm2e}
\SetAlCapSkip{1em}

\newtheorem{thm}{Theorem}[section]
\newtheorem{prop}[thm]{Proposition}

\newnumbered{mdef}[thm]{Definition}
\newnumbered{rem}[thm]{Remark}

\newcommand{\st}{\;\middle\vert\;}

\usepackage{tikz}

\title{Minimal mutation-infinite quivers}
\author{John Lawson}
\date{\today}
\extraline{%
The author's PhD studies are supported by an EPSRC studentship.}
\classno{}
\doi{}
\makeatletter
\renewcommand{\@journal}{}
\renewcommand{\@footnotetwo}{}
\makeatother

\begin{document}
\maketitle
\begin{abstract}
	Quivers constructed from hyperbolic Coxeter simplices give examples of minimal
	mutation-infinite quivers, however they are not the only such quivers. We
	classify minimal mutation-infinite quivers through a number of moves and link
	the representatives of the classes with the hyperbolic Coxeter simplices, plus
	exceptional classes which are not related to simplices.
\end{abstract}

\tableofcontents
\section{Introduction}\label{sec:intro}
Mutations on quivers were introduced by Fomin and Zelevinsky in their
introduction to cluster algebras in 2002~\cite{FZ-CA1}. Since then this area has
been widely studied, with applications in numerous areas of mathematics.

The mutation class of a quiver is the collection of all quivers which can be
obtained from the original through a sequence of mutations. All finite sized
mutation classes were classified by Felikson, Shapiro and
Tumarkin in~\cite{FST-FiniteMutation}, this classification necessarily contains
the mutation-classes of quivers that give finite-type cluster algebras, which
were classified by Fomin and Zelevinsky in~\cite{FZ-CA2}. This classification
states that all finite-type cluster algebras come from quivers given by
orientations of Dynkin diagrams, while quivers from orientations of affine Dynkin
diagrams are also mutation-finite.

In their classification Fomin and Zelevinsky introduced mutations on diagrams,
and Seven classified all minimal 2-infinite diagrams
in~\cite{Seven-MinInfinite}. The work by Seven on minimal 2-infinite diagrams
inspired the study of minimal mutation-infinite quivers and this paper builds on
work done by Felikson, Shapiro and Tumarkin in~\cite[Section
7]{FST-FiniteMutation} proving a number of useful results about minimal
mutation-infinite quivers.

Minimal mutation-infinite quivers are those which belong to an infinite mutation
class, but any subquiver belongs to a finite mutation
class.  Simply-laced diagrams from hyperbolic Coxeter simplices of finite volume
have the property that any subdiagram is a Dynkin or affine Dynkin diagram and
so any mutation-infinite orientation of such a diagram is minimal
mutation-infinite.  The motivating question behind this study is whether the
family of minimal mutation-infinite quivers from orientations of hyperbolic
Coxeter simplex diagrams contains all minimal mutation-infinite quivers.

In this paper we classify all minimal mutation-infinite quivers, with classes
represented by orientations of hyperbolic Coxeter simplex diagrams as well as
some exceptional representatives. The classification is defined in terms of
moves, which are specific sequences of mutations. In general, mutation does not
preserve the property of a quiver being minimal mutation-infinite, however the moves are
constructed in such a way that they do.

{%
\renewcommand{\thethm}{\ref{thm:smallcase}}%
\begin{thm}
	Any minimal mutation-infinite quiver with at most $9$ vertices can be
	transformed through sink-source mutations and at most 5 moves to one of an
	orientation of a hyperbolic Coxeter diagram, a double arrow quiver or an
	exceptional quiver.
\end{thm}
\addtocounter{thm}{-1}%
}

{%
\renewcommand{\thethm}{\ref{thm:main}}%
\begin{thm}
	Any minimal mutation-infinite quiver can be transformed through sink-source
	mutations and at most 10 moves to one of an orientation of a hyperbolic
	Coxeter diagram, a double arrow quiver or an exceptional quiver.
\end{thm}
\addtocounter{thm}{-1}%
}

The results of this paper give a procedure to check whether any given quiver is
mutation-infinite without having to compute any part of its mutation class. This
procedure follows from the fact that any mutation-infinite quiver must contain a
minimal mutation-infinite complete subquiver.

In Section~\ref{sec:mutations} of this paper we remind the reader of the process
of mutating quivers, and recall the properties arising from mutation-equivalence
of quivers. Using these definitions we introduce minimal mutation-infinite
quivers and highlight the interest behind their study.
In Section~\ref{sec:simplices} we recall the relations between quivers, diagrams
and Coxeter simplices, as well as constructing quivers from orientations of
certain Coxeter diagrams given by these simplices. Some examples of these
quivers give minimal mutation-infinite quivers.

Section~\ref{sec:class} introduces a classification of all minimal
mutation-infinite quivers through a number of elementary moves defined in
Section~\ref{sec:moves} and listed in Appendix~\ref{sec:move_list}. These moves
allow minimal mutation-infinite quivers to be transformed to other minimal
mutation-infinite quivers and so admit a classification of such quivers.

The quiver classification involved a large computational effort to find all
minimal mutation-infinite quivers. Appendix~\ref{sec:compute} details the
procedures used in this computation. Details about implementations of these
procedures and the complete lists of minimal mutation-infinite quivers can be
found on the author's website~\cite{Lawson-MMI}.

\begin{acknowledgements}
The author would like to thank John Parker and Pavel Tumarkin for their
supervision and support.
\end{acknowledgements}
\section{Mutations}\label{sec:mutations}

The following gives an introduction to mutations of quivers.  Further
information and the extension to cluster algebras can be found in introductory
survey articles such as~\cite{Williams-CA} or~\cite{Keller-Intro}.  Given a
graph denote a cycle of length one as a loop, a cycle of length two as a
2-cycle.

\begin{mdef}
	A \textbf{quiver} is an oriented graph with possibly more than one arrow
	between any two vertices. In the following a quiver is always considered with
	the additional restriction that it contains no loops or 2-cycles.
\end{mdef}

This restriction ensures that the quiver is uniquely determined by its
skew-symmetric adjacency (or exchange) matrix. This correspondence depends on an
indexing of the vertices of the quiver and it is convenient to always consider
the quiver with such a numbering, so that any vertex can be referred to by its
index.

\begin{mdef}\label{def:mutation}
	\textbf{Mutation} of a quiver is a function at a vertex $k$ of the quiver which
	changes the arrows around the vertex according to 3 rules:
	\begin{enumerate}
		\item Whenever there is a path through vertex $k$ of the form $i \to k \to
			j$ then add an arrow $i \to j$.
		\item Reverse the direction of all arrows adjacent to $k$
		\item Remove a maximal collection of 2-cycles created in this process.
	\end{enumerate}

	Figure~\ref{fig:mutex} shows an example of mutation on a quiver.
\end{mdef}
\begin{figure}
	\includegraphics{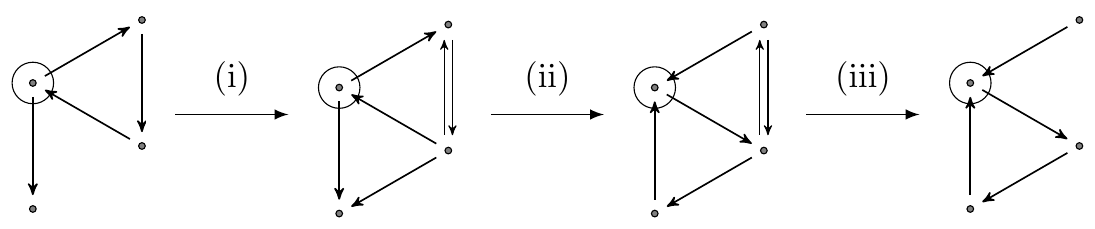}
	\caption{Example of mutation at the circled vertex. Each step in
	Definition~\ref{def:mutation} is shown separately.}
	\label{fig:mutex}
\end{figure}

\begin{mdef}
	Two quivers $P$ and $Q$ are \textbf{mutation-equivalent} if there exists a
	sequence of mutations taking $P$ to $Q$.
	The \textbf{mutation-class} of a quiver is the equivalence class under this
	equivalence relation.
	A quiver is \textbf{mutation-finite} if it belongs to a mutation-class of
	finite size, otherwise the quiver is \textbf{mutation-infinite}.
\end{mdef}

All mutation-finite quivers have been classified by Felikson, Shapiro and
Tumarkin in~\cite{FST-FiniteMutation} as either a quiver arising from an
orientation of a triangulation of a surface or a quiver in one of 11 exceptional
mutation-classes.

\subsection{Partial ordering on quivers}

A partial ordering can be put on all quivers given by inclusion of complete
subquivers.

\begin{figure}
	\includegraphics{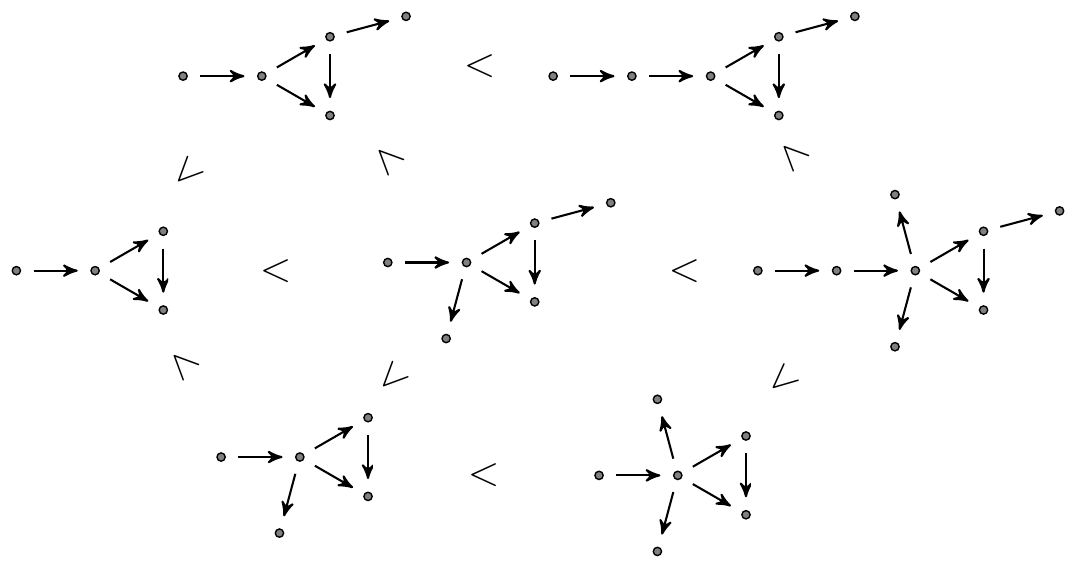}
	\caption{The partial ordering on some examples of quivers.}
	\label{fig:po}
\end{figure}

\begin{mdef}
	Given two quivers $P$ and $Q$, then $P < Q$ if $P$ can be obtained by removing
	vertices (and all arrows adjacent to each removed vertex) from $Q$.
	Equivalently, if $B_P$ and $B_Q$ are the exchange matrices of $P$ and $Q$
	respectively, then $P < Q$ if $B_P$ is a submatrix of $B_Q$ up to
	simultaneously permuting the rows and columns of $B_P$.
	If $P < Q$ then call $P$ a \textbf{complete subquiver} of $Q$.
\end{mdef}

For brevity it is convenient to omit the word complete and write subquiver to
mean complete subquiver.  Denote the vertices of $Q$ as $u_1,\dotsc,u_
m,v_1,\dotsc,v_n$ and let $P$ be the subquiver of $Q$ obtained by removing
vertices $v_1,\dotsc,v_n$.  Then any mutation at $u_i$ commutes with removing
these vertices $\lbrace v_j \rbrace$, giving the following proposition.

\begin{center}
\begin{tikzpicture}[]
	\node (Q) {$Q$};
	\node[right of=Q,node distance=4cm] (P){$P$};
	\node[below of=Q,node distance=1.5cm] (uQ){$\mu_{u_i}(Q)$};
	\node[below of=P,node distance=1.5cm] (uP){$\mu_{u_i}(P)$};
	\draw[->] (Q) -- (P) node[midway,above]{remove $\lbrace v_j\rbrace$};
	\draw[->] (Q) -- (uQ);
	\draw[->] (P) -- (uP);
	\draw[->] (uQ) -- (uP) node[midway,below]{remove $\lbrace v_j\rbrace$};
\end{tikzpicture}
\end{center}

\begin{prop}\label{prop:subquivers}
	A quiver which contains some mutation-infinite quiver as a subquiver is
	necessarily mutation-infinite. Equivalently any subquiver of a
	mutation-finite quiver is mutation-finite.
\end{prop}

Proposition~\ref{prop:subquivers} shows that there are minimal mutation-infinite
quivers with respect to the above partial ordering. Equivalently these minimal
mutation-infinite quivers could be defined as follows:

\begin{mdef}
	A \textbf{minimal mutation-infinite quiver} is a mutation-infinite quiver for
	which every subquiver is mutation-finite.
\end{mdef}

\subsection{Properties of minimal mutation-infinite quivers}\label{sec:mmi-props}

In their paper on the classification of mutation-finite quivers Felikson,
Shapiro and Tumarkin prove a useful fact about minimal mutation-infinite
quivers.

\begin{thm}[{\cite[Lemma 7.3]{FST-FiniteMutation}}]\label{thm:mmi_size}
	Any minimal mutation-infinite quiver contains at most 10 vertices.  
	Equivalently, any mutation-infinite quiver of size greater than 10 must
	contain a mutation-infinite subquiver.
\end{thm}

\begin{figure}
	\hspace*{\fill}%
	\hbox{\includegraphics{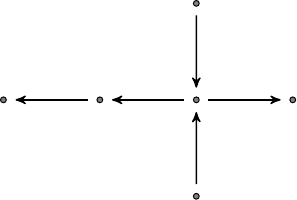}}%
	\hfill%
	\hbox{\includegraphics{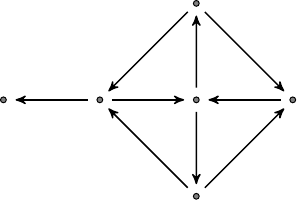}}%
	\hfill%
	\hbox{\includegraphics{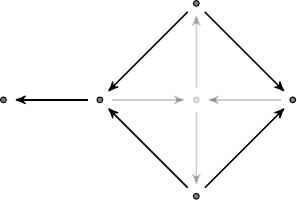}}%
	\hspace*{\fill}%
	\caption{The left quiver is minimal mutation-infinite. Mutation at the central
		vertex yields the central quiver however this is not minimal
		mutation-infinite. Removing the central vertex gives the right quiver, which
		is also mutation-infinite.}
	\label{fig:mmi-mut}
\end{figure}

An important restriction of the minimal mutation-infinite property of quivers is
that it is not preserved by mutation.  An example of a mutation which does not
preserve the minimal mutation-infinite property of a quiver is given in
Figure~\ref{fig:mmi-mut}. However there are some specific mutations which do
preserve this property, of which sink-source mutations are an example.

\begin{mdef}
	A \textbf{sink} is a vertex in a quiver such that all adjacent arrows are
	directed into that vertex, whereas a \textbf{source} is a vertex such that all
	adjacent arrows are directed away from the vertex.
	Define a \textbf{sink-source mutation} as a mutation at either a sink or a source.
\end{mdef}

\begin{prop}\label{prop:ss-mut}
	Sink-source mutations of a quiver preserve whether it is minimal
	mutation-infinite or not.
\end{prop}
\begin{proof}
	A mutation at a sink (resp.\ source) reverses the direction of all arrows
	adjacent to it, so the vertex becomes a source (sink).

	Let $P$ be a minimal mutation-infinite quiver, and $Q$ a quiver obtained from
	$P$ by a sink-source mutation at a vertex $v$. The quiver $Q$ is
	mutation-equivalent to $P$, so is mutation-infinite. The subquiver of $Q$
	obtained by removing the mutated vertex $v$ is precisely the same as the
	subquiver of $P$ constructed by removing $v$. Any other subquiver of $Q$ is a
	single sink-source mutation away from the corresponding subquiver of $P$.
	Every subquiver of $P$ is mutation-finite, so every subquiver of $Q$ is also
	mutation-finite, hence $Q$ is minimal mutation-infinite.
\end{proof}

The following well known fact limits the possible quivers which could be minimal
mutation-infinite.

\begin{prop}\label{prop:3-mut-inf}
	A mutation-finite quiver with at least 3 vertices has at most two arrows
	between any two vertices.
\end{prop}

A comprehensive proof of this fact can be found in Derksen and Owen's
paper~\cite[Section 3]{DO-Finite}.  This is equivalent to stating that any
quiver with 3 or more arrows between any two vertices is necessarily
mutation-infinite.

Every subquiver of a minimal mutation-infinite quiver is mutation-finite and so
each subquiver has at most 2 arrows between any two vertices. Therefore the
minimal mutation-infinite quiver itself has at most 2 arrows between any two
vertices.

\begin{prop}\label{prop:3-vert-mmi}
	Any mutation-infinite quiver with 3 vertices is minimal mutation-infinite.
\end{prop}
\begin{proof}
	All quivers with only 2 vertices are mutation-finite, as mutation at either
	vertex just reverses the direction of the arrows. Hence the mutation-class
	contains just these two quivers.
\end{proof}

\section{Coxeter simplices}\label{sec:simplices}

It is known that hyperbolic Coxeter simplices of finite volume exist up to
dimension 9 and so admit diagrams with up to 10 vertices. In the following
section we explore the links between these diagrams and the minimal
mutation-infinite quivers which also exist with up to 10 vertices, as stated in
Theorem~\ref{thm:mmi_size}.

An $n$-dimensional Coxeter simplex is considered in one of three spaces:
spherical, Euclidean and hyperbolic. As a simplex they are the convex hull of
$n+1$ points and so have $n+1$ facets. 

\begin{mdef}
	A simplex is a \textbf{Coxeter simplex} if the hyperplanes which make up the
	faces have dihedral angles all submultiples of $\pi$. In the case of
	hyperbolic Coxeter simplices we allow the case where the planes meet at the
	boundary and so have dihedral angle $0$.
\end{mdef}

Given a Coxeter simplex we denote the hyperplanes by $H_i$ and the angle
between hyperplanes $H_i$ and $H_j$ by $\frac{\pi}{k_{ij}}$.

\begin{mdef}
	The \textbf{Coxeter diagram} associated to a Coxeter simplex is an unoriented
	graph with a vertex $i$ for each hyperplane $H_i$ and a weighted edge between
	vertices $i$ and $j$ when $k_{ij} > 3$ with weight $k_{ij}$. We add an
	unweighted edge between $i$ and $j$ when $k_{ij} = 3$, and if the angle
	between two hyperplanes $H_i$ and $H_j$ is $\frac{\pi}{2}$ then no edge is
	put between $i$ and $j$.

	In the hyperbolic case, where two hyperplanes meet at the boundary, then the
	edge is given weight $\infty$.
\end{mdef}

The Coxeter group associated to a given Coxeter diagram is constructed from the
following representation, where each generator $s_i$ represents reflection in
the hyperplane $H_i$,
\[ \left\langle s_i \st s_i^2 = 1 = \left( s_i s_j \right)^{k_{ij}}
\right\rangle. \]

\subsection{Simply-laced Coxeter simplex diagrams in different spaces}

\begin{mdef}
	\textbf{Simply-laced} Coxeter diagrams are those for which $k_{ij} \in
	\left\lbrace 2, 3 \right\rbrace$ for all $i$ and $j$.
\end{mdef}

This is equivalent to only allowing angles of $\frac{\pi}{2}$ and
$\frac{\pi}{3}$ in the Coxeter simplex.
Simply-laced Coxeter diagrams only contain edges with no weights, and so a
quiver can be constructed from the diagram by choosing an orientation for each
edge.

Coxeter simplices can be considered over spherical, Euclidean or hyperbolic
space. In each case the quivers obtained by choosing an orientation for the
simply-laced Coxeter diagrams have different properties. The following are
well known results about the spherical and Euclidean cases.

\begin{rem}\label{rem:cox-mutfin}
	In~\cite{Coxeter-ReflGroups}, Coxeter classified simply-laced spherical
	Coxeter simplex diagrams as Dynkin diagrams of type $A$, $D$ and $E$.
	Orientations of these diagrams are mutation-finite quivers and give
	finite-type cluster algebras, as shown in Fomin and Zelevinsky's
	classification of finite-type cluster algebras~\cite{FZ-CA2}.

	Similarly, simply-laced Euclidean Coxeter simplex diagrams are affine Dynkin
	diagrams of type $\tilde{A}$, $\tilde{D}$ and $\tilde{E}$. Felikson, Shapiro
	and Tumarkin's mutation-finite classification~\cite{FST-FiniteMutation} shows
	that orientations of these diagrams are mutation-finite but give infinite-type
	cluster algebras.
\end{rem}

It is known that the hyperbolic Coxeter simplex diagrams satisfy the following
property.

\begin{rem}\label{rem:hcs-subdiags}
	Any subdiagram of a simply-laced hyperbolic Coxeter simplex diagram is either
	a Dynkin or an affine Dynkin diagram.
\end{rem}

This follows from Theorems 3.1 and 3.2 of Vinberg's paper~\cite{Vinberg-HypGroups}
concerning the reflection groups generated by the reflections in $n$ hyperplanes
of an $n$ dimensional hyperbolic Coxeter simplex.

\subsection{A family of minimal mutation-infinite quivers}

Given a simply-laced hyperbolic Coxeter simplex diagram, construct a quiver by
choosing an orientation on each edge. From Remark~\ref{rem:hcs-subdiags},
any subquiver of this quiver will be an orientation of either a Dynkin diagram
or an affine Dynkin diagram and so Remark~\ref{rem:cox-mutfin}
shows that any subquiver is mutation-finite.

Using the classification of mutation-finite quivers given
in~\cite{FST-FiniteMutation}, it can be seen that almost all orientations of
hyperbolic Coxeter simplex diagrams are mutation-infinite. There is precisely
one mutation-finite orientation of hyperbolic Coxeter simplex diagrams for each
size $k$ between $5$ and $9$, with two of size $4$, as shown in
Figure~\ref{fig:fin-hcs}.

\begin{figure}
	\includegraphics{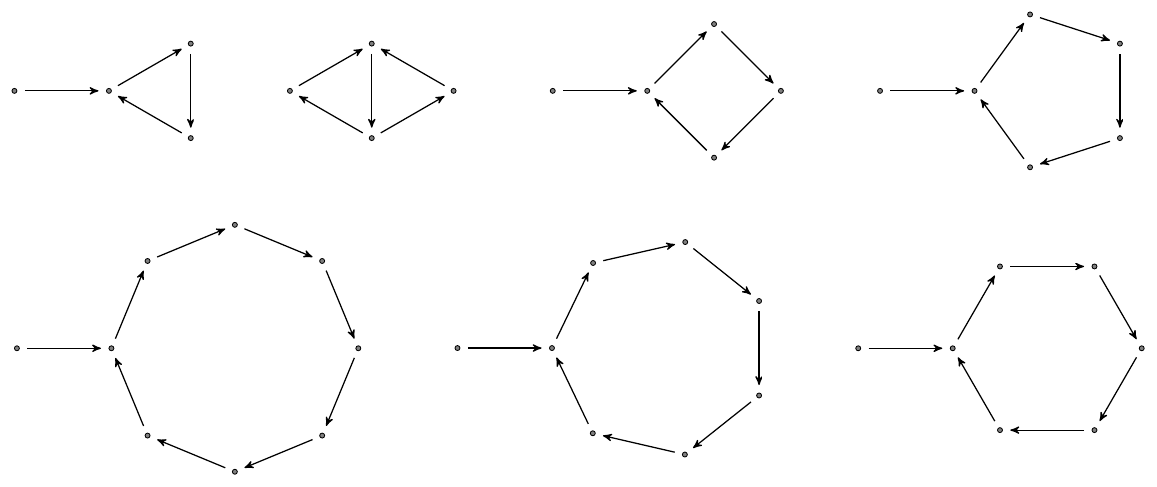}
	\caption{Mutation-finite orientations of hyperbolic Coxeter simplex diagrams}
	\label{fig:fin-hcs}
\end{figure}

It follows from Remarks~\ref{rem:cox-mutfin}
and~\ref{rem:hcs-subdiags} that all mutation-infinite orientations of
hyperbolic Coxeter simplex diagrams are in fact minimal mutation-infinite
quivers. This then raises the question of whether all minimal mutation-infinite
quivers can be given in this form or not.

\begin{prop}\label{prop:nonhcs}
	There exist minimal mutation-infinite quivers which are not orientations of a
	hyperbolic Coxeter simplex diagram for all sizes of quiver from 5 to 10.
\end{prop}
\begin{figure}
	\hspace*{\fill}
	\raisebox{-0.5\height}{\includegraphics{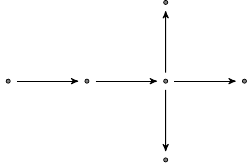}}%
	\hfill%
	\raisebox{-0.5\height}{\includegraphics{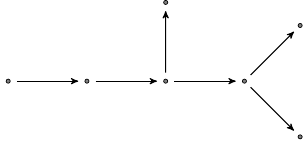}}%
	\hfill%
	\raisebox{-0.5\height}{\includegraphics{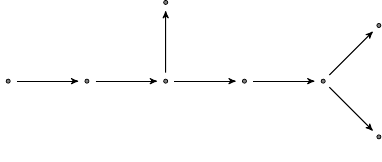}}%
	\hspace*{\fill}
	\caption{Orientations of tree-like hyperbolic Coxeter simplex diagrams of size
	$6$, $7$ and $8$.}
	\label{fig:tree-like}
\end{figure}
\begin{figure}
	\hspace*{\fill}
	\raisebox{-0.5\height}{\includegraphics{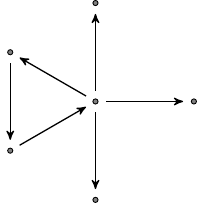}}%
	\hfill%
	\raisebox{-0.5\height}{\includegraphics{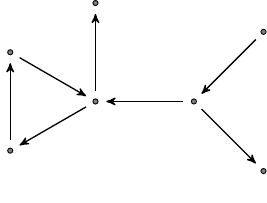}}%
	\hfill%
	\raisebox{-0.5\height}{\includegraphics{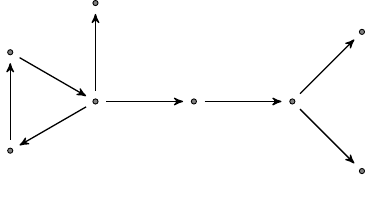}}%
	\hspace*{\fill}
	\caption{Minimal mutation-infinite quivers not orientations of hyperbolic
		Coxeter simplex diagrams.}
	\label{fig:nonhcs}
\end{figure}
\begin{proof}
	To prove this it suffices to give an example of such a quiver for each size.
	The construction of this quiver for size $6 \leq k \leq 10$ is as follows:

	Take the tree-like hyperbolic Coxeter simplex diagram of size $k$ and orient
	it in such a way that all arrows point the same way as illustrated in
	Figure~\ref{fig:tree-like}. This quiver $Q$ is minimal mutation-infinite as
	shown above, and contains an orientation of the Dynkin diagram $A_3$ as a
	subquiver. Mutating at the centre vertex of this $A_3$ creates an oriented
	triangle in the resulting quiver $P$, giving the quivers in
	Figure~\ref{fig:nonhcs} which are not orientations of hyperbolic Coxeter
	simplex diagrams.

	The resulting quiver is mutation-equivalent to the orientation of a hyperbolic
	Coxeter simplex, so is mutation-infinite. Each subquiver obtained by removing
	vertex $n$ from $P$ is either the same as the subquiver obtained by removing
	$n$ from $Q$, or a single mutation away from it. Hence as $Q$ is minimal
	mutation-infinite, all such subquivers are mutation-finite and so $P$ is also
	minimal mutation-infinite.

	The only minimal mutation-infinite quivers with 5 vertices are of the form
	shown in Figure~\ref{fig:nonhcs-5}. Mutation of an orientation of a hyperbolic
	Coxeter simplex diagram gives such a quiver, and all subquivers are
	mutation-equivalent to subquivers of the initial quiver so the resulting quiver
	is again minimal mutation-infinite.
\end{proof}
\begin{figure}
	\hspace*{\fill}
	\raisebox{-0.5\height}{\includegraphics{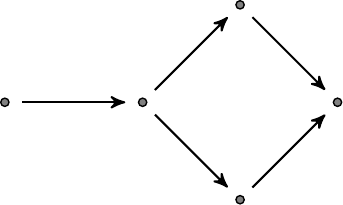}}%
	\hfill%
	\raisebox{-0.5\height}{\includegraphics{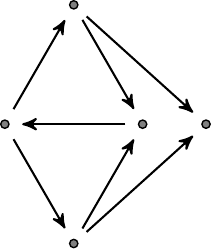}}%
	\hspace*{\fill}
	\caption{The quiver on the left is an orientation of a hyperbolic Coxeter
	simplex diagram. Mutation at the trivalent vertex yields the quiver on the
	right which is also minimal mutation-infinite.}
	\label{fig:nonhcs-5}
\end{figure}

\section{Minimal mutation-infinite quiver moves}\label{sec:moves}

Orientations of hyperbolic Coxeter simplex diagrams give a family of minimal
mutation-infinite quivers, however Proposition~\ref{prop:nonhcs} shows the
existence of other minimal mutation-infinite quivers. This section discusses the
approach taken to classify all such quivers.

Many examples of minimal mutation-infinite quivers are only a small number of
mutations away from an orientation of a hyperbolic Coxeter diagram. As discussed
in Section~\ref{sec:mmi-props} mutations do not in general preserve the minimal
mutation-infinite property of a quiver, however it can be proved that specific
mutations, where a vertex is surrounded by a particular subquiver, do indeed
preserve this property. An example of such a mutation was used in the proof of
Proposition~\ref{prop:nonhcs}. These particular mutations which preserve the
minimal mutation-infinite property can be considered as \textbf{moves} among all
minimal mutation-infinite quivers.

As mutation acts by changing the quiver locally around the mutated vertex, while
leaving arrows further from the vertex fixed, these moves can be defined in
terms of the subquivers which change under the mutations. In this way applying the move
is equivalent to replacing some subquiver with a different subquiver.

The minimal mutation-infinite preserving mutations often depend on some
restriction of how the vertices in the subquivers are connected in the whole
quiver outside the subquiver. This data then needs to be encoded in the
moves along with the subquivers.

\begin{mdef}
	When referred to in a move, a \textbf{line} is a line of vertices such that one
	end point is connected to the move subquiver. A line of length zero consisting
	of just a single vertex is also considered valid.
\end{mdef}

\subsection{A move example}

\begin{figure}
	\includegraphics{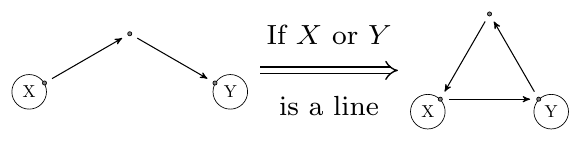}
	\caption{An example of a minimal mutation-infinite move.}
	\label{fig:movexample}
\end{figure}

Figure~\ref{fig:movexample} gives an example of one such move. The move is
applied to a quiver by mutating at the central vertex. The circles labelled $X$
and $Y$ denote connected components of the quiver fixed by the move. The
vertex on the boundary of $X$ is considered to be contained in $X$. In this
case the move requires that one of the components be a line (or just the single
vertex) for the move to apply. Figure~\ref{fig:moveapp} shows some examples of
quivers for which this move is applicable or not and Figure~\ref{fig:movexapp}
shows how it acts on the first quiver in Figure~\ref{fig:moveapp}.

\begin{figure}
	\hspace*{\fill}
	\raisebox{-0.5\height}{\includegraphics{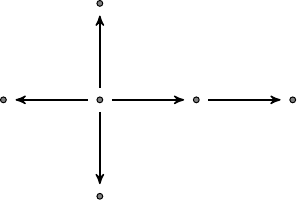}}%
	\hfill%
	\raisebox{-0.5\height}{\includegraphics{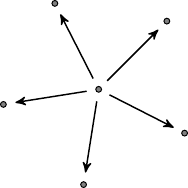}}%
	\hfill%
	\raisebox{-0.5\height}{\includegraphics{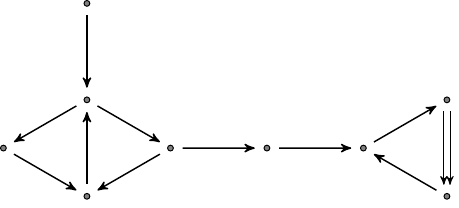}}%
	\hspace*{\fill}
	\caption{The move in Figure~\ref{fig:movexample} applies to the first quiver,
		but does not apply to the others. The second quiver does not contain either
		subquiver, while the third does, but neither component is a line with an
		endpoint in the subquiver.}
	\label{fig:moveapp}
\end{figure}
\begin{figure}
	\includegraphics{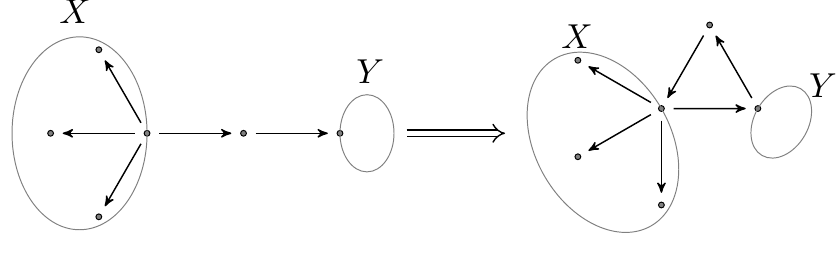}
	\caption{An example of how the move in Figure~\ref{fig:movexample} changes a
	quiver. Note that $Y$ consists of a single vertex, and so can be thought of as
	a line of length zero.}
	\label{fig:movexapp}
\end{figure}

\begin{prop}
	The image of a minimal mutation-infinite quiver under the move in
	Figure~\ref{fig:movexample} is minimal mutation-infinite.
\end{prop}
\begin{proof}
	Let $P$ be the initial quiver and $Q$ its image under the move. Denote the
	vertex at which mutation occurs during the move as $w$.

	The move is equivalent to mutation at $w$ hence $Q$ is in the same mutation
	class as $P$. $P$ is mutation-infinite so $Q$ is also mutation-infinite.

	Any subquiver $Q'$ of $Q$, obtained by removing a vertex $v$ in either $X$ or
	$Y$, will contain $w$. Mutating at $w$ will yield a quiver $\mu(Q')$ that is
	equal to one obtained by removing the corresponding vertex $v$ from $P$. As
	$P$ is minimal mutation-infinite, such a subquiver of $P$ is necessarily
	mutation-finite, hence $\mu(Q')$ is mutation-finite and so $Q'$ is also
	mutation-finite.

	Removing $w$ gives a subquiver $Q'$ of $Q$ which is not mutation-equivalent to
	a subquiver of $P$. Instead the extra condition that either $X$ or $Y$ is a
	line ensures that this quiver is a subquiver of $P$ by removing the vertex at
	the end of that line, and so is mutation-finite. For example consider the
	quivers in Figure~\ref{fig:movexapp}, removing $w$ from $Q$  gives a quiver
	which is the same as one obtained by removing the vertex in $Y$ from $P$.

	Hence $Q$ is minimal mutation-infinite.
\end{proof}

The proofs for all moves are similar to this. The moves are always constructed
from sequences of mutations, so the image is mutation-infinite and the quivers
obtained by removing vertices outside those vertices which are mutated by the
move can always be mutated back to a subquiver of the initial quiver. The
challenge is determining whether a quiver obtained by removing a vertex at which
one of the mutations took place is mutation-equivalent to a subquiver of the
initial quiver.

\begin{prop} The move given by reversing the move in Figure~\ref{fig:movexample}
is a valid move.
\end{prop}
\begin{proof}
	As discussed above, it suffices to show that removing the vertex at which the
	mutation occurs yields a mutation-finite quiver. Denote the initial quiver as
	$P$, the image $Q$ and the mutated vertex $w$.

	Removing $w$ from $Q$ gives a quiver $R$ which is the disjoint union of $X$ and
	$Y$, therefore $R$ is mutation-finite if and only if both $X$ and $Y$ are.

	Both $X$ and $Y$ are contained in $P$, so are subquivers of $P$ and hence are
	mutation-finite. Therefore $R$ is also mutation-finite, so $Q$ is minimal
	mutation-infinite.
\end{proof}

Appendix~\ref{sec:move_list} contains a list of all moves necessary to classify
minimal mutation-infinite quivers.

The moves required to classify all minimal mutation-infinite quivers up to size
9 only have requirements that certain components are lines or are connected to
other components by lines. For the size 10 quivers, stricter conditions are
required as some moves require that a certain quiver constructed from
the components is mutation-finite. The quivers constructed in such a way are
always of a smaller size and so the results for smaller size quivers can be
applied.

\begin{figure}
	\includegraphics{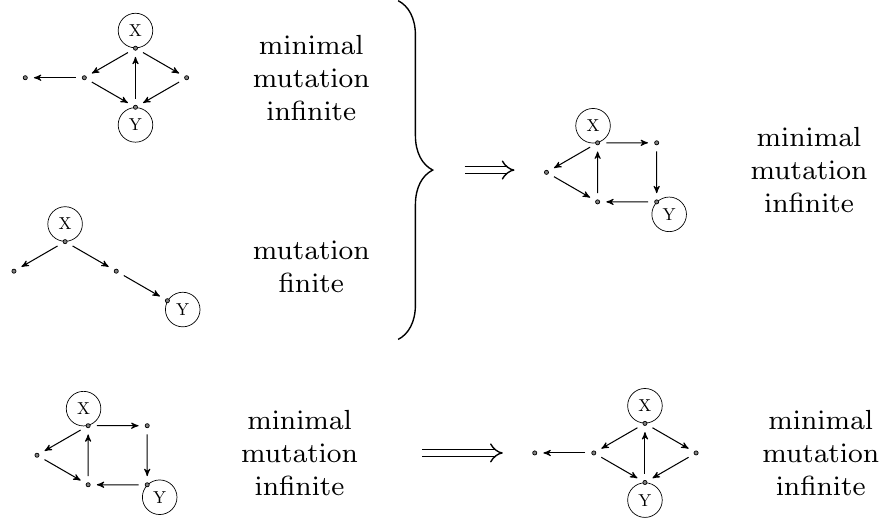}
	\caption{Example of a size 10 move with added constraints}
	\label{fig:10move}
\end{figure}

Figure~\ref{fig:10move} gives an example of one such move for size 10 quivers.
In one direction the move applies without any additional constraints, but in the
other direction the move requires that a certain quiver constructed from quiver
components is mutation-finite.

\newcommand{\hrepspace}{\hspace{2em}}
\newcommand{\vrepspace}{1em}
\newcommand{\dvrepspace}{3em}
\newcommand{\addrep}[1]{\raisebox{-0.5\height}{\includegraphics{#1}}}

\begin{table}[p]
	{Rank 4:}\hspace*{\fill}
	\\
	\addrep{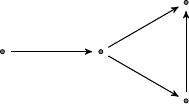}\hrepspace%
	\addrep{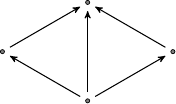}\hrepspace%
	\addrep{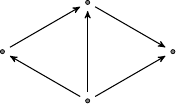}\hrepspace%
	\addrep{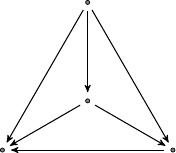}\hrepspace%
	\addrep{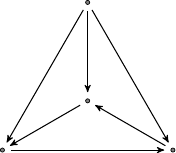}\hrepspace%
	\addrep{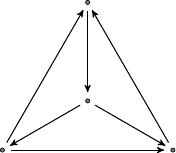}\hrepspace%
	\\[\vrepspace]
	{Rank 5:}\hfill
	\addrep{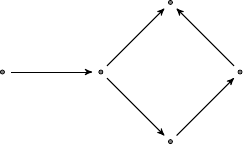}\hrepspace%
	\addrep{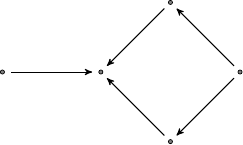}\hrepspace%
	\addrep{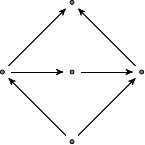}\hrepspace%
	\addrep{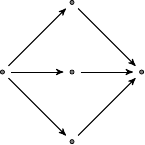}\hrepspace%
	\hspace*{\fill}
	\\[\vrepspace]
	{Rank 6:}\hfill
	\addrep{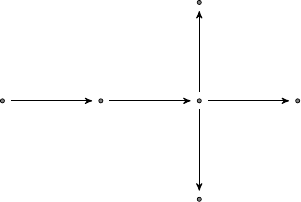}\hrepspace%
	\addrep{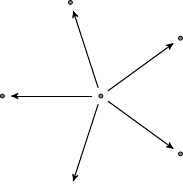}\hrepspace%
	\addrep{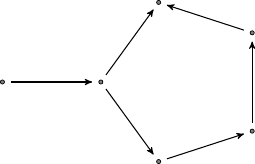}\hrepspace%
	\addrep{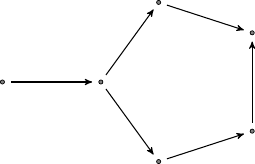}\hrepspace%
	\hspace*{\fill}
	\\[\vrepspace]
	{Rank 7:}\hfill
	\addrep{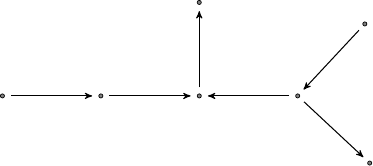}\hrepspace%
	\hspace*{\fill}
	\\[\vrepspace]
	\addrep{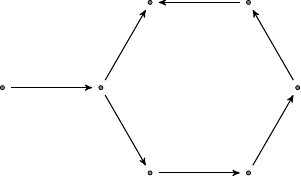}\hrepspace%
	\addrep{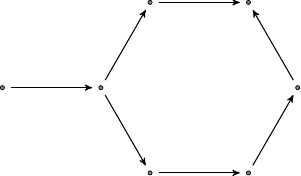}\hrepspace%
	\addrep{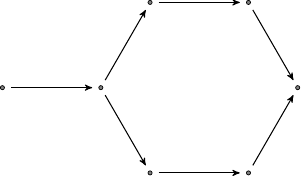}\hrepspace%
	\\[\vrepspace]
	{Rank 8:}\hfill
	\addrep{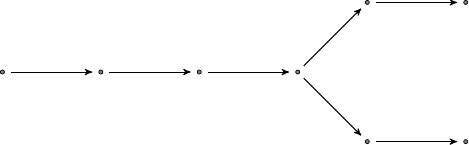}\hrepspace%
	\addrep{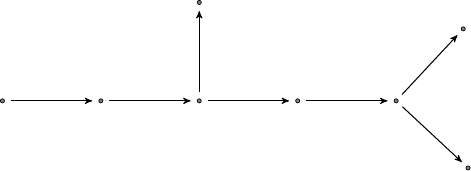}\hrepspace%
	\hspace*{\fill}
	\\[\vrepspace]
	\addrep{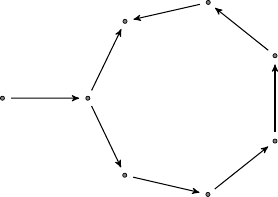}\hrepspace%
	\addrep{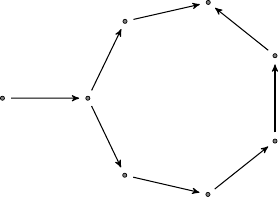}\hrepspace%
	\addrep{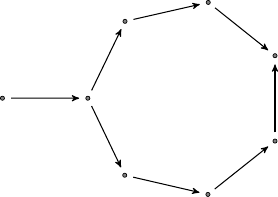}\hrepspace%
	\\[\vrepspace]
	{Rank 9:}\hfill
	\addrep{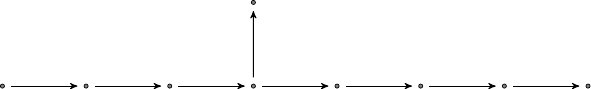}\hrepspace%
	\addrep{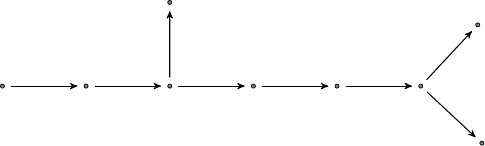}\hrepspace%
	\hspace*{\fill}
	\\[\vrepspace]
	\addrep{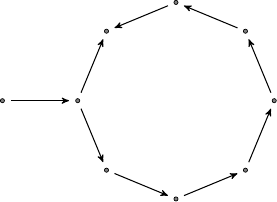}\hrepspace%
	\addrep{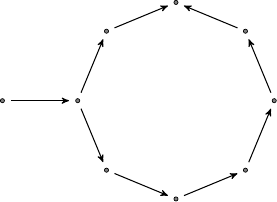}\hrepspace%
	\addrep{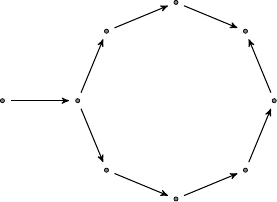}\hrepspace%
	\addrep{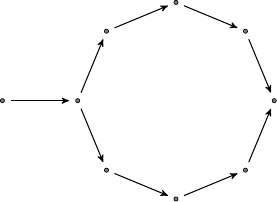}\hrepspace%
	\\[\vrepspace]
	{Rank 10:}\hfill
	\addrep{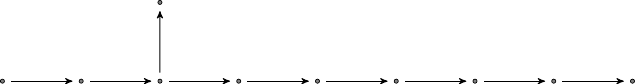}\hrepspace%
	\addrep{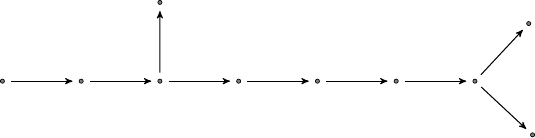}%
	\hspace*{\fill}
	\\[\vrepspace]
\caption{Representatives: Orientations of hyperbolic Coxeter simplex
diagrams}\label{reps:hcs}
\end{table}

\begin{table}[p]
	\addrep{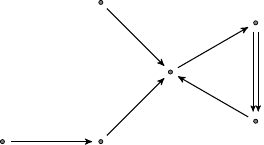}\hrepspace%
	\addrep{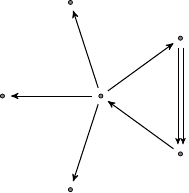}\hrepspace%
	\addrep{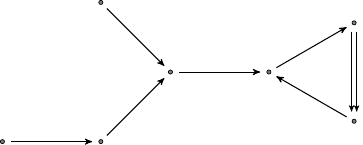}\hrepspace%
	\\[\vrepspace]
	\addrep{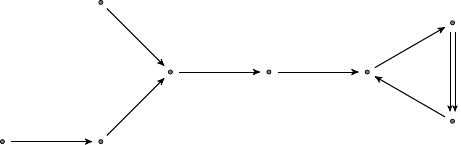}\hrepspace%
	\addrep{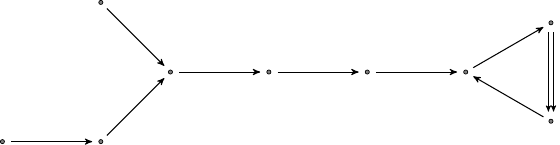}\hrepspace%
	\\[\vrepspace]
	\addrep{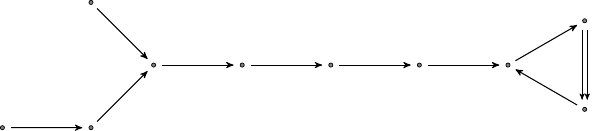}%
	\\[\vrepspace]
	\caption{Representatives: Double arrow quivers}\label{reps:da}
\end{table}

\begin{table}[p]
	\addrep{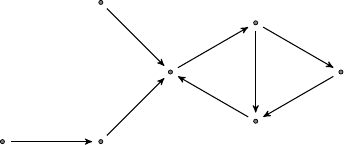}\hrepspace%
	\addrep{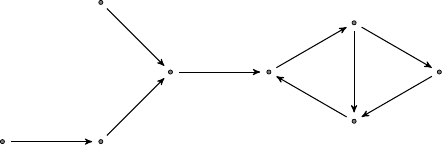}\hrepspace%
	\\[\vrepspace]
	\addrep{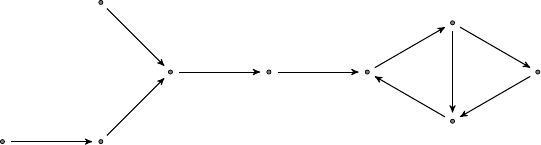}\hrepspace%
	\addrep{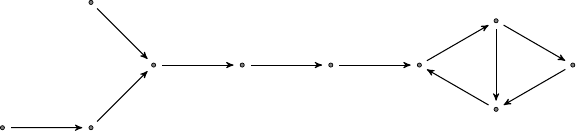}\hrepspace%
	\\[\dvrepspace]
	\addrep{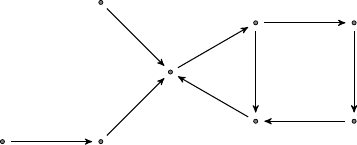}\hrepspace%
	\addrep{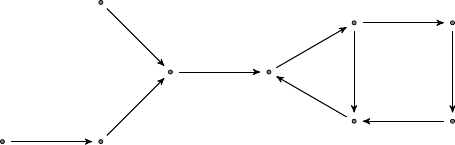}\hrepspace%
	\\[\vrepspace]
	\addrep{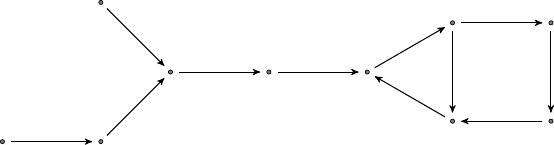}\hrepspace%
	\\[\dvrepspace]
	\addrep{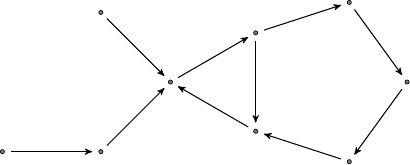}\hrepspace%
	\addrep{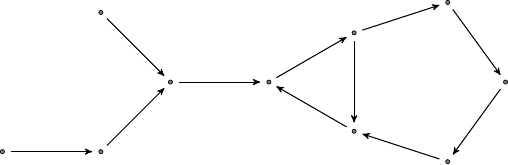}\hrepspace%
	\\[\dvrepspace]
	\addrep{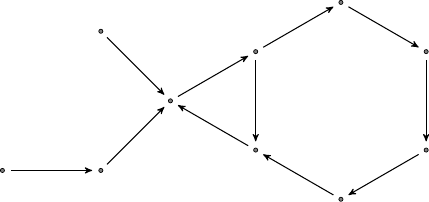}%
	\\[\vrepspace]
	\caption{Representatives: Exceptional quivers}\label{reps:exc}
\end{table}

\section{Classifying minimal mutation-infinite quivers}\label{sec:class}

Define an equivalence relation where two quivers are equivalent if one quiver
can be obtained from the other through a sequence of moves. Then these moves
together with the list of representatives (see Tables~\ref{reps:hcs},~\ref{reps:da}
and~\ref{reps:exc}) classify minimal mutation-infinite quivers.

Hyperbolic Coxeter simplex diagrams give a family of minimal mutation-infinite
quivers, and so orientations of these diagrams are some of the representatives
of the classes. In~\cite[Corollary 4]{CK-TriCatsToCA} Caldero and Keller proved
that any two acyclic orientations of a diagram, belonging to the same mutation
class, are mutation-equivalent through a sequence of sink-source mutations. As a
result of this, given any hyperbolic Coxeter simplex diagram the classification
requires a representative for each acyclic orientation which can not be obtained
from any other acyclic orientation by sink-source mutations.

Many minimal mutation-infinite quivers can be transformed into one of the
hyperbolic Coxeter diagrams, however there are some which can not. Therefore the
classification contains hyperbolic Coxeter classes and some exceptional classes.
A particular case of these exceptional cases arises from those minimal
mutation-infinite quivers which contain a double arrow between two vertices.
There are two such classes for quivers of size $6$ and one class for each size
between $7$ and $10$.

The result places a bound on the number of moves required to transform any
minimal mutation-infinite quiver to one of the class representatives. Diagrams
of the representatives can be found in Tables~\ref{reps:hcs},~\ref{reps:da}
and~\ref{reps:exc}. This statement can then be reversed to give a construction
of all possible minimal mutation-infinite quivers from these representatives.
The procedure to do this would be progressively applying the moves to the set of
all quivers computed so far. As the number of moves is bounded this procedure
will stop and at that point all minimal mutation-infinite quivers will have been
computed.

\begin{thm}\label{thm:smallcase}
	Any minimal mutation-infinite quiver with at most $9$ vertices can be
	transformed through sink-source mutations and at most 5 moves to one of an
	orientation of a hyperbolic Coxeter diagram, a double arrow quiver or an
	exceptional quiver (see Tables~\ref{reps:hcs}--\ref{reps:exc}).
\end{thm}

As discussed in Section~\ref{sec:moves} above the moves required for quivers of
size 10 have more constraints and are more complicated than those for smaller
quivers. As such this result needs to be restated when considering these larger
quivers.

\begin{thm}\label{thm:main}
	Any minimal mutation-infinite quiver can be transformed through sink-source
	mutations and at most 10 moves to one of an orientation of a hyperbolic
	Coxeter diagram, a double arrow quiver or an exceptional quiver (see
	Tables~\ref{reps:hcs}--\ref{reps:exc}).
\end{thm}

Appendix~\ref{sec:compute} discusses the computations used to verify this result
and find all minimal mutation-infinite quivers. There are in total 18,799 such
quivers (excluding those with 3 vertices) which are orientations of 574
different graphs. Pictures of all minimal mutation-infinite quivers organised
into their move-classes can be found on the author's website~\cite{Lawson-MMI}.

\subsection{A mutation-infinite check using minimal mutation-infinite quivers}
Any mutation-infinite quiver must contain some minimal mutation-infinite quiver
as a subquiver. Hence given a list of all the minimal mutation-infinite
quivers there is an algorithm to check whether a given quiver is
mutation-infinite without having to compute any part of its mutation class.

Let $Q$ be a possibly mutation-infinite quiver and $\lbrace P_i \rbrace_{i\in %
I}$ be all minimal mutation-infinite quivers indexed by $I$. For each $i
\in I$ if $P_i$ is a subquiver of $Q$ then $Q$ is mutation-infinite,
otherwise continue to the next $i$. If no minimal mutation-infinite quiver is in
fact a subquiver of $Q$ then $Q$ is mutation-finite.

\bibliographystyle{plain}
\bibliography{main}
\clearpage
\appendix
\section{Computing minimal mutation-infinite quivers}\label{sec:compute}
The quiver classification involved a large computational effort to find all
minimal mutation-infinite quivers. This section details the procedures used in
this computation. Details about implementations of these procedures can be found
on the author's website~\cite{Lawson-MMI}.

\subsection{Finding the size of a mutation-class}

An important computation in all the following algorithms is determining whether
a given quiver is mutation-finite or mutation-infinite. There is a fast
approximation which can prove a quiver is mutation-infinite and a slower
procedure which proves a quiver is mutation-finite.

\subsubsection*{Fast approximation to check whether a quiver is mutation-infinite}

Proposition~\ref{prop:3-mut-inf} states that any mutation-finite quiver has at
most 2 arrows between any two vertices. This gives a procedure that can prove
that a quiver is mutation-infinite, but which cannot prove that a quiver is
mutation-finite.  This procedure was used in computations by Felikson, Shapiro
and Tumarkin in their classification of skew-symmetric mutation-finite
quivers~\cite{FST-FiniteMutation} and Shapiro's comments on the procedure can be
found on his website~\cite{Shapiro-Programs}.

The procedure, given in Algorithm~\ref{algo:fast}, checks whether the quiver
contains 3 or more arrows between any two vertices, if it does then the quiver
is mutation-infinite.  Otherwise, pick a vertex at random and mutate the quiver
at this vertex and repeat with this new quiver.

\begin{algorithm}[htbp]
	\DontPrintSemicolon
	\KwIn{$Q$ Quiver to check}
	\KwData{$M$ Number of mutations to perform}
	\KwData{$k$ Counter initially 0}
	\KwResult{Whether $Q$ is mutation-infinite, or probably mutation-finite}

	\BlankLine
	\While{$k < M$}{%
		\If{$Q$ contains 3 or more arrows between 2 vertices}{%
			\KwRet{$Q$ is mutation-infinite}
		}
		Choose a random vertex\;
		Mutate $Q$ at this vertex\;
		Increment $k$\;
	}
	\KwRet{$Q$ is probably mutation-finite}
	\caption{Fast approximation whether a quiver is mutation-infinite}
	\label{algo:fast}
\end{algorithm}

For mutation-finite quivers this process would never terminate without the bound
on the number of mutations, and it is possible that for mutation-infinite
quivers the randomly chosen mutations never generate an edge with more than 2
arrows. Therefore this is only an approximation and a maximum number of mutations
should be attempted before stopping.  If no quiver was found with more than 2
arrows between two vertices then, provided the number of mutations was high, the
quiver is probably mutation-finite.

\subsubsection*{Computing a full finite mutation-class}

While the above procedure can show a quiver is probably mutation-finite, we
require a procedure that can definitively prove it. To do this the whole
mutation-class of the quiver must be found.

\begin{algorithm}[htbp]
	\DontPrintSemicolon
	\KwIn{$Q$ Quiver}
	\KwData{$L$ Queue of quivers to mutate}
	\KwData{$A$ List of all quivers found in the mutation class so far}
	\KwData{$M_P$ For each quiver $P$, a map taking a vertex in $P$ to the quiver
	obtained by mutating $P$ at that vertex (if the mutation has been computed)}
	\KwResult{$A$ List of all quivers in the mutation class}
	\BlankLine
	Add $Q$ to $L$\;
	\While{$L$ is not empty}{%
		Remove quiver $P$ from the top of queue $L$\;
		\For{$i = 1$ \KwTo (Number of vertices)}{%
			\eIf{$M_P$ has a quiver at vertex $i$}{%
				Continue to next vertex\;
			}{%
				Let $P'$ be the mutation of $P$ at $i$\;
				\eIf{$P' \in A$}{%
					Update $M_{P'}$ so vertex $i$ points to $P$\;
				}{%
					Create $M_{P'}$ with vertex $i$ pointing to $P$\;
					Add $P'$ to $A$\;
					Add $P'$ to $L$\;
				}
				Update $M_P$ so vertex $i$ points to $P'$\;
			}
		}
	}
	\KwRet{$A$}
	\caption{Compute mutation-class of a mutation-finite quiver}
	\label{algo:mut-class}
\end{algorithm}

The algorithm to find the mutation-class of a mutation-finite quiver calculates
the whole exchange graph of the initial quiver. First compute all mutations of
this quiver, then for each of these quivers compute all mutations and continue
until no further quivers are computed.  By keeping track of which mutations link
two vertices, only those mutations which are not known need to be computed.  See
Algorithm~\ref{algo:mut-class}.

\subsubsection*{Slower mutation-finite check}

The above algorithm will only terminate if the initial quiver is mutation-finite. In
the case of a mutation-infinite quiver, the mutation-class is infinite, so the
computation will continue indefinitely. The algorithm can be adapted to
terminate for mutation-infinite quivers using the result in
Proposition~\ref{prop:3-mut-inf}.

Once a new quiver is computed which has not yet been found, check whether it
contains three or more arrows between any two vertices. If it does then the
mutation-class is known to be infinite, so the procedure can be terminated. See
Algorithm~\ref{algo:slow-inf}.

\begin{algorithm}[htbp]
	\DontPrintSemicolon
	\KwIn{$Q$ Quiver}
	\KwData{$L$ Queue of quivers to mutate}
	\KwData{$A$ List of all quivers found in the mutation class so far}
	\KwData{$M_P$ For each quiver $P$, a map taking a vertex in $P$ to the quiver
	obtained by mutating $P$ at that vertex (if the mutation has been computed)}
	\KwResult{Whether $Q$ is mutation-infinite or not}
	\BlankLine
	Add $Q$ to $L$\;
	\While{$L$ is not empty}{%
		Remove quiver $P$ from the top of queue $L$\;
		\For{$i = 1$ \KwTo (Number of vertices)}{%
			\eIf{$M_P$ has a quiver at vertex $i$}{%
				Continue to next vertex\;
			}{%
				Let $P'$ be the mutation of $P$ at $i$\;
				\eIf{$P' \in A$}{%
					Update $M_{P'}$ so vertex $i$ points to $P$\;
				}{%
					\If{$P'$ has more than 3 arrows between 2 vertices}{%
						\KwRet{$Q$ is mutation-infinite}
					}
					Create $M_{P'}$ with vertex $i$ pointing to $P$\;
					Add $P'$ to $A$\;
					Add $P'$ to $L$\;
				}
				Update $M_P$ so vertex $i$ points to $P'$\;
			}
		}
	}
	\KwRet{$Q$ is mutation-finite}
	\caption{Determine whether a quiver is mutation-infinite or not}
	\label{algo:slow-inf}
\end{algorithm}

There are only a finite number of ways to draw a graphs with a fixed number
of vertices and up to 2 arrows between any two vertices. Hence in an infinite
mutation-class there will eventually be a quiver with 3 or more arrows between
two vertices and therefore the procedure will always terminate.

The two procedures to compute whether a quiver is mutation-finite can be
combined to provide a faster run time in the majority of cases. By first using
the fast approximation, most mutation-infinite quivers will be identified as
mutation-infinite and any quivers which are not then get passed to the slower
check to confirm whether they are mutation-finite.

\subsection{Computing quivers}

The above algorithms give procedures to tell whether a given quiver is
mutation-finite or mutation-infinite. By iterating through a range of quivers these
checks can be used to find all quivers which satisfy certain properties.

\subsubsection*{Computing all mutation-finite quivers}

Proposition~\ref{prop:subquivers} states that all subquivers of a
mutation-finite quiver are again mutation-finite. This fact is used to build up
mutation-finite quivers of a certain size $n$ by adding vertices to the
mutation-finite quivers with $n-1$ vertices. All $2$ vertex quivers are
mutation-finite, so with these as a starting point we can recursively
compute all mutation-finite quivers of size $n$, using the procedure in
Algorithm~\ref{algo:finite}.

By Proposition~\ref{prop:3-mut-inf} any mutation-finite quiver contains at most
2 arrows between any two vertices, so when adding a vertex to the quivers of
size $n-1$ it suffices to only add either $0$, $1$ or $2$ between the new vertex
and any others. Adding more arrows would immediately yield a mutation-infinite
quiver.

\SetKwFunction{FnFinite}{Finite}
\SetKwProg{Fn}{Function}{}{}
\begin{algorithm}[htbp]
	\DontPrintSemicolon
	\KwIn{$n$ Size of quiver to output}
	\KwData{$A$ List of mutation-finite quivers}
	\KwResult{A list of all mutation-finite quivers of size $n$}
	\BlankLine
	\Fn{\FnFinite{size $n$}}{%
		\If{$n = 2$}{%
			Let $A = \left\lbrace \cdot\rightarrow\cdot, \cdot\rightrightarrows\cdot
			\right\rbrace$\;
			\KwRet{$A$}
		}
		\ForEach{Quiver $Q$ in \FnFinite{$n-1$}}{%
			\ForEach{Extension of $Q$ to possibly mutation-finite quiver $Q'$}{%
				\If{$Q'$ is mutation-finite}{%
					Add $Q'$ to $A$\;
				}
			}
		}
		\KwRet{$A$}
	}
	\caption{Compute all mutation-finite quivers of size $n$}
	\label{algo:finite}
\end{algorithm}

\subsubsection*{Computing all minimal mutation-infinite quivers}

Any subquiver of a minimal mutation-infinite quiver is a mutation-finite quiver.
Therefore to construct these quivers of a certain size $n$ start with all
mutation-finite quivers of size $n-1$ and extend the quiver by adding another
vertex in all possible ways with either $0$, $1$ or $2$ arrows between the new
vertex and any other vertices. The quivers obtained in this way could then be minimal
mutation-infinite and so this needs to be verified.

For a given quiver to be minimal mutation-infinite it must satisfy two
conditions, namely that it is mutation-infinite and that every subquiver is
mutation-finite. Both of these conditions can be checked using the above
procedures.

\subsection{Checking number of moves}

Theorem~\ref{thm:main} states the maximum number of moves required to transform
any minimal mutation-infinite quiver to one of the class representatives. To
compute this number each minimal mutation-infinite quiver is checked in turn to
find the minimal number of moves needed to transform that quiver to its class
representative.

This minimal number of moves can be found by applying all applicable moves to the
initial quiver and storing the number of moves taken to reach each quiver obtained
in this way. 
We can ensure that the number of moves used to obtain a class representative is
minimal by always choosing the next quiver used in the process to be the one
obtained through the fewest number of moves.

\newenvironment{nospacecenter}{%
\centering%
}{%
\par%
}
\section{List of moves}\label{sec:move_list}
This section lists all moves required to transform any minimal mutation-infinite
quiver to one of the representatives. Any listed move should also be considered
along with the move where all arrows are reversed.

Where a move has the requirement that one of the components is a line this
requires that the component is a line with one of its endpoints adjacent to the
move subquiver.

\subsection{Moves for quivers of size 5}\hspace*{\fill}~\\
\begin{nospacecenter}
	\includegraphics{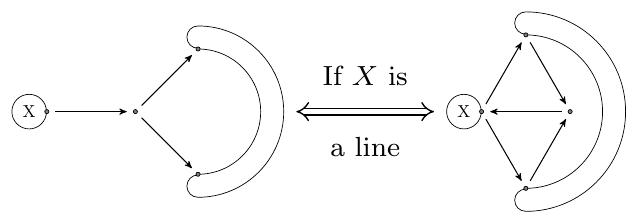}
\end{nospacecenter}

\subsection{Additional moves for quivers of size 6}\hspace*{\fill}~\\
\begin{multicols}{2}
\begin{nospacecenter}
	\includegraphics{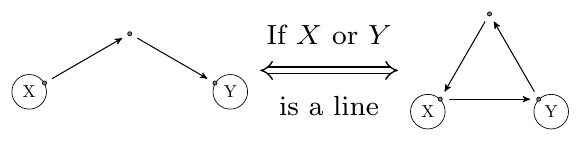}\\
	\includegraphics{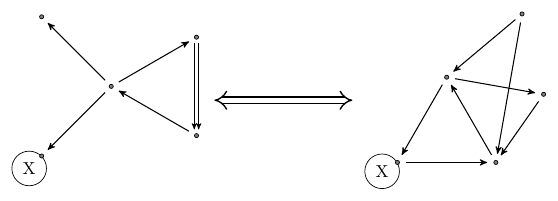}\\
\end{nospacecenter}
\end{multicols}

\subsection{Additional moves for quivers of size 7}\hspace*{\fill}~\\
\begin{multicols}{2}
\begin{nospacecenter}
	\includegraphics{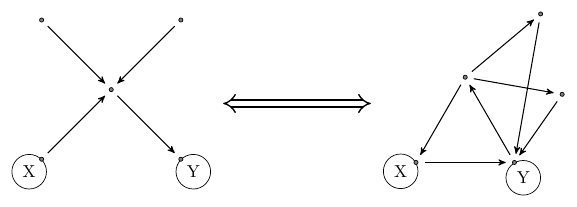}\\
	\includegraphics{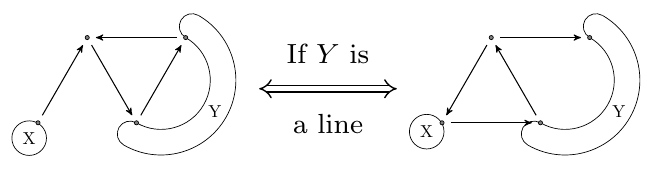}\\
	\includegraphics{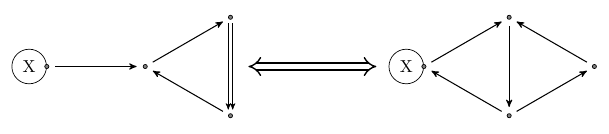}\\
	\includegraphics{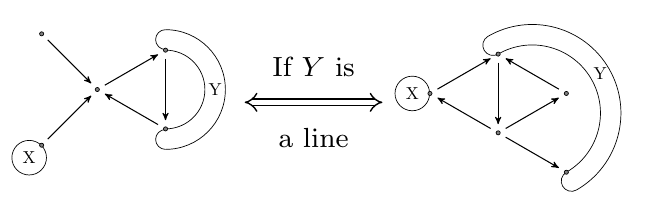}\\
\end{nospacecenter}
\end{multicols}

\subsection{Additional moves for quivers of size 8}\hspace*{\fill}~\\
\begin{multicols}{2}
\begin{nospacecenter}
	\includegraphics{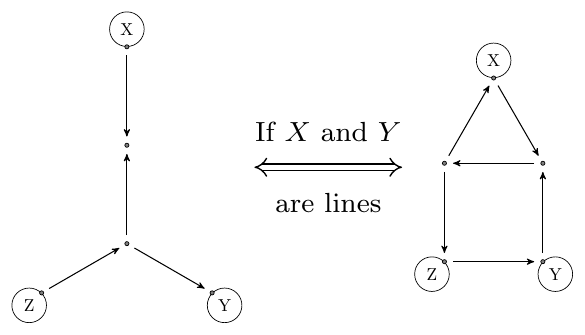}\\
	\includegraphics{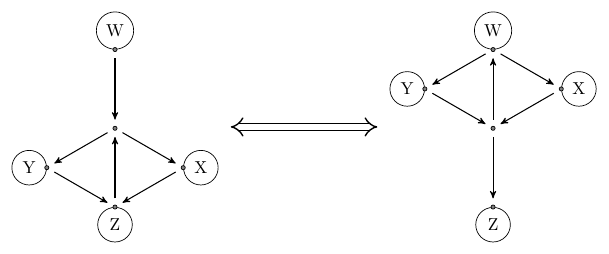}\\
	\includegraphics{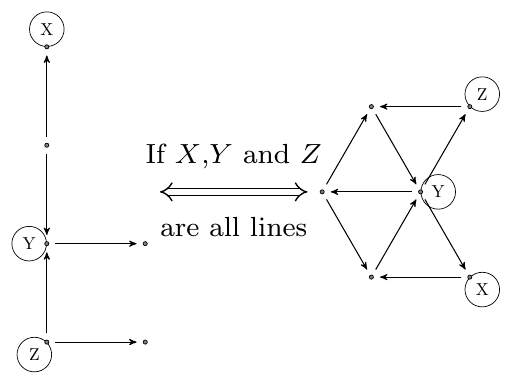}\\
	\includegraphics{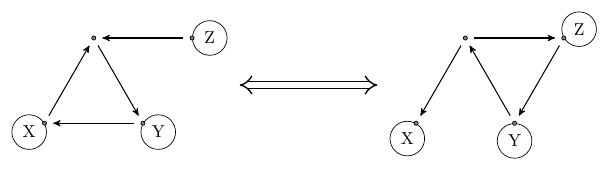}\\
\end{nospacecenter}
\end{multicols}

\subsection{Additional moves for quivers of size 9}\hspace*{\fill}~\\
\begin{multicols}{2}
\begin{nospacecenter}
	\includegraphics{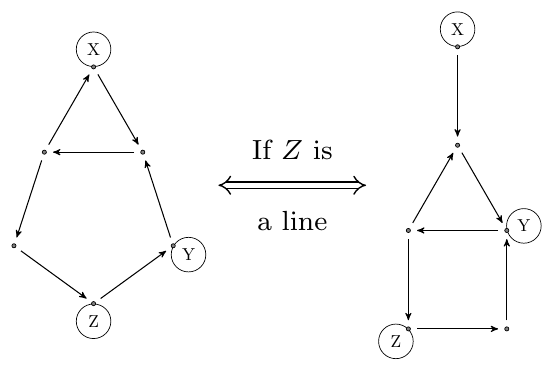}\\
	\includegraphics{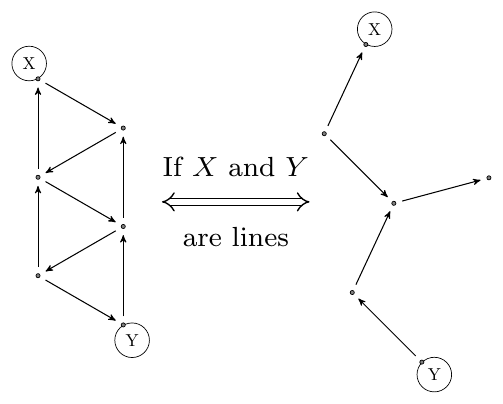}\\
	\includegraphics{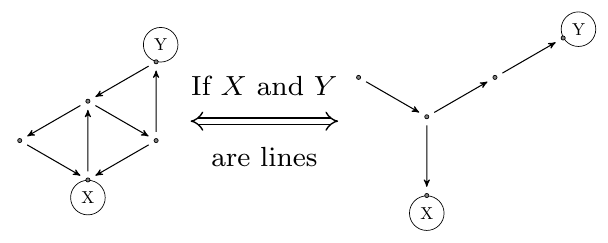}\\
	\includegraphics{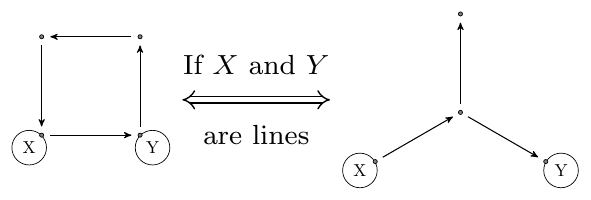}\\
	\includegraphics{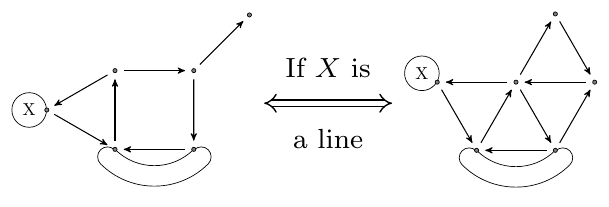}\\
	\includegraphics{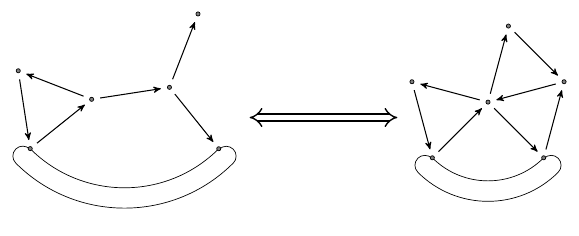}\\
	\includegraphics{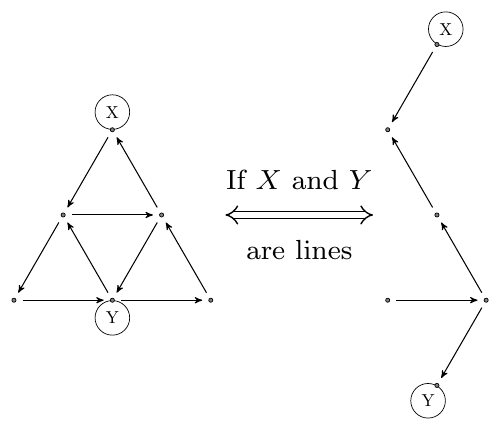}\\
	\includegraphics{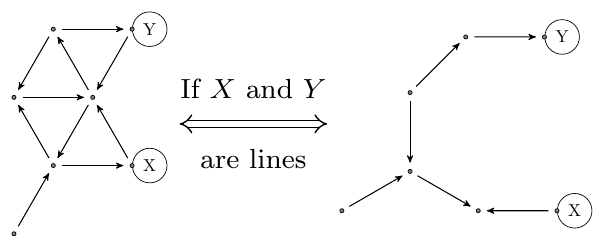}\\
	\includegraphics{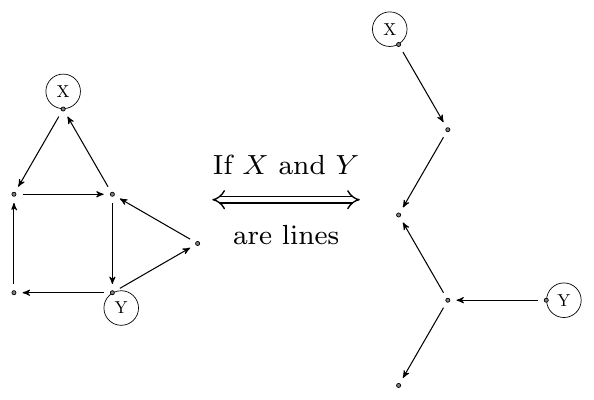}\\
	\includegraphics{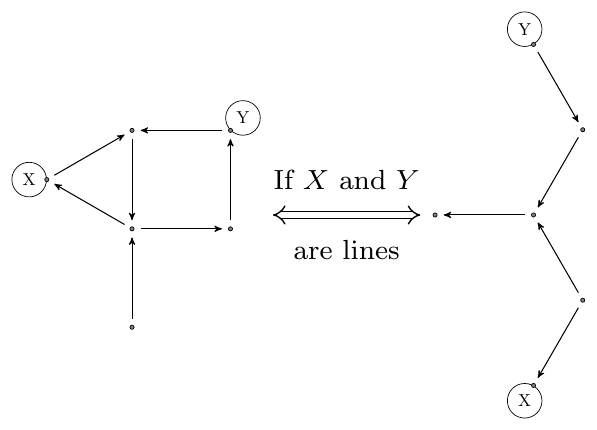}\\
	\includegraphics{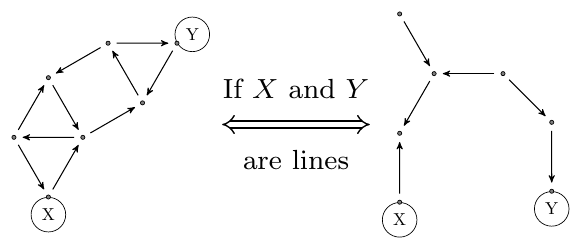}\\
	\includegraphics{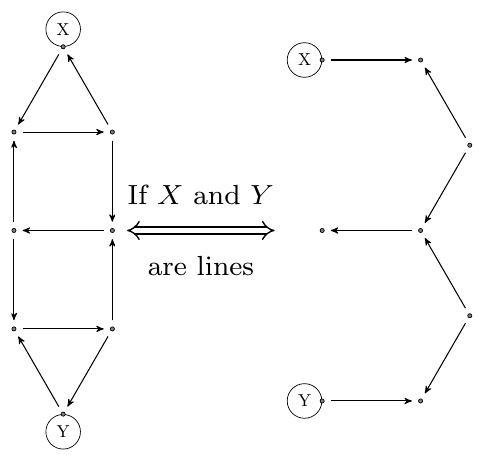}\\
\end{nospacecenter}
\end{multicols}

\subsection{Additional moves for quivers of size 10}\hspace*{\fill}~\\
\begin{multicols}{2}
\begin{nospacecenter}
	\includegraphics{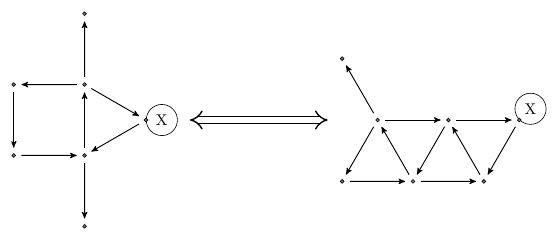}\\
	\includegraphics{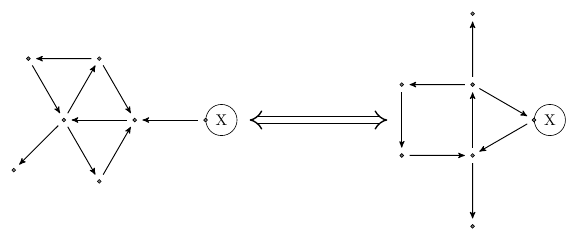}\\
	\includegraphics{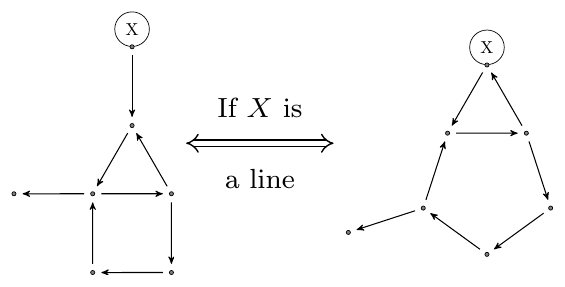}\\
	\includegraphics{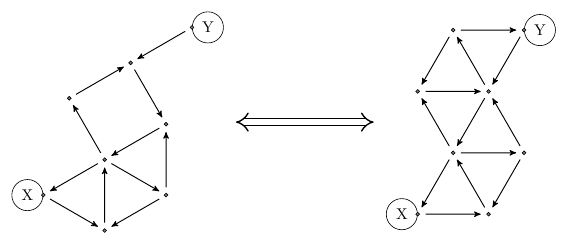}\\
	\includegraphics{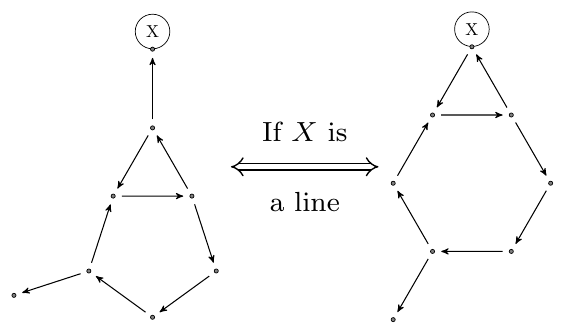}\\
	\includegraphics{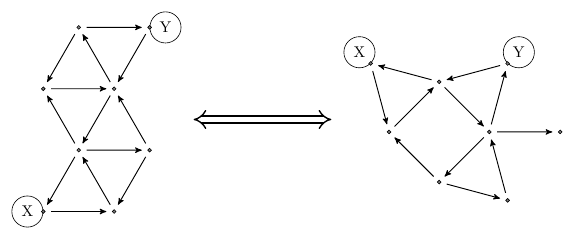}\\
\end{nospacecenter}%
\end{multicols}%
\begin{nospacecenter}
	\includegraphics{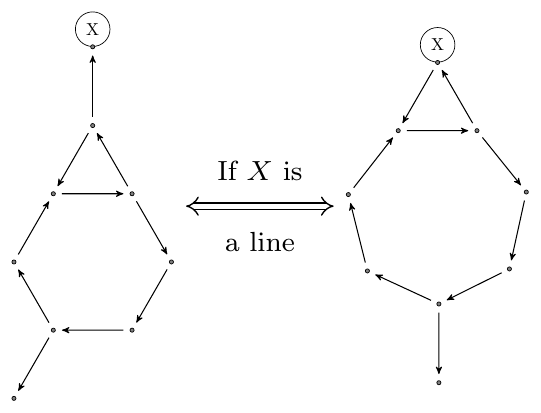}\quad
	\includegraphics{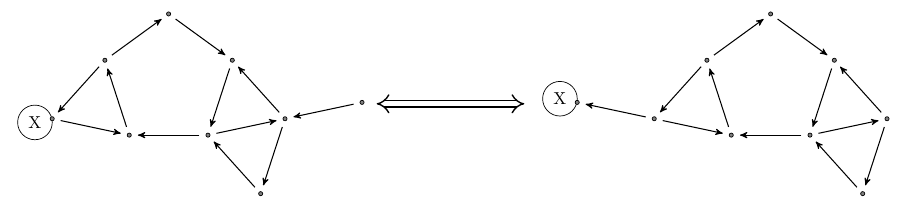}\quad
	\includegraphics{10mov1}\quad
	\includegraphics{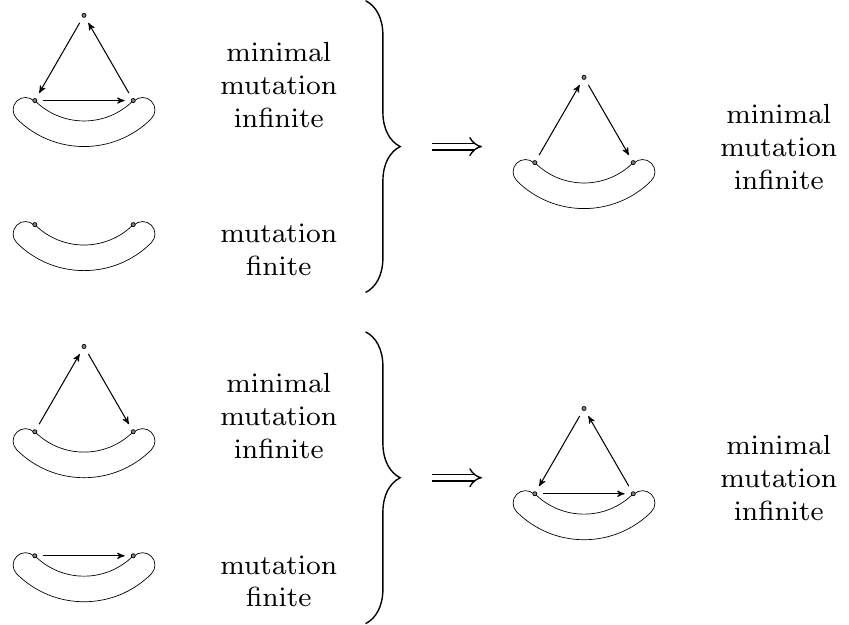}\quad
	\includegraphics{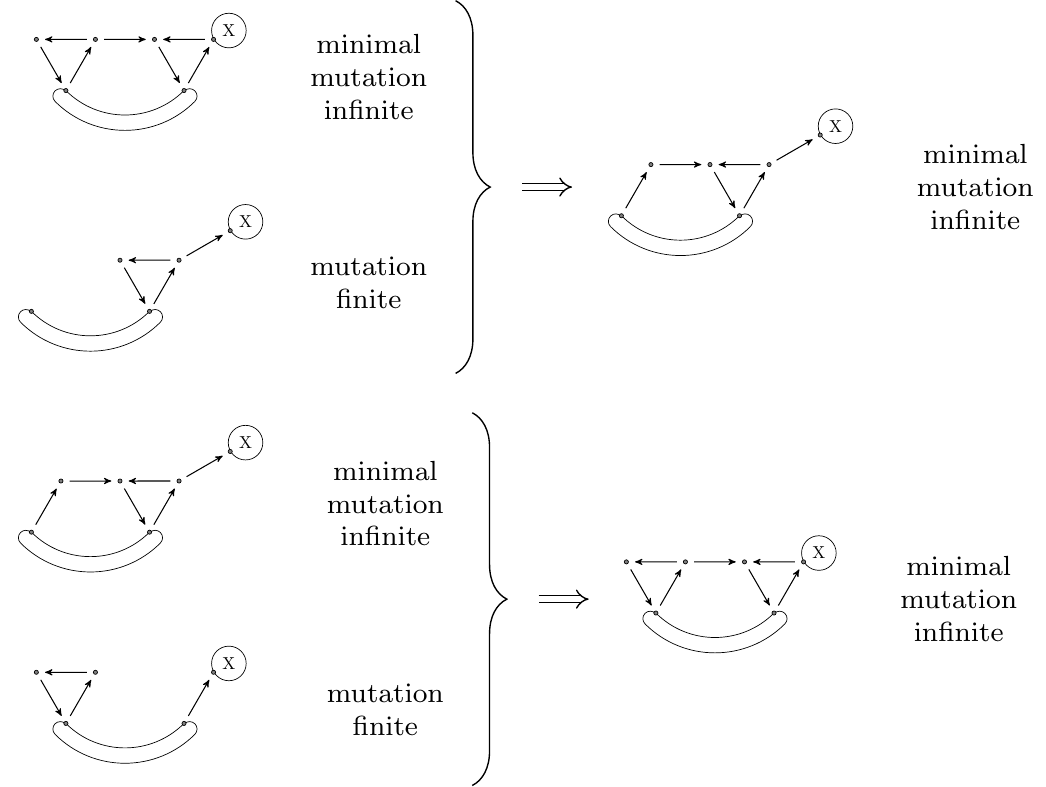}\quad
	\includegraphics{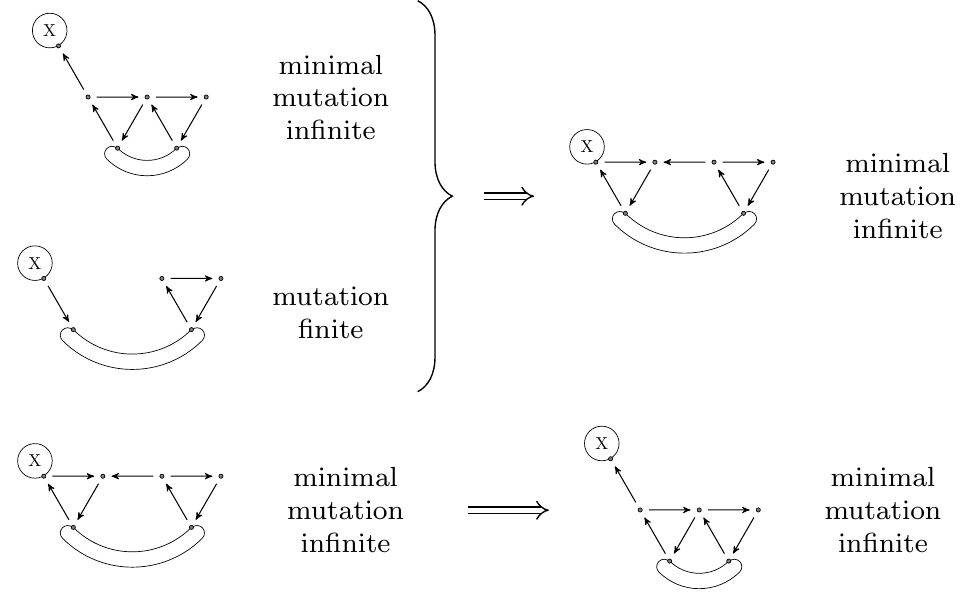}\quad
	\includegraphics{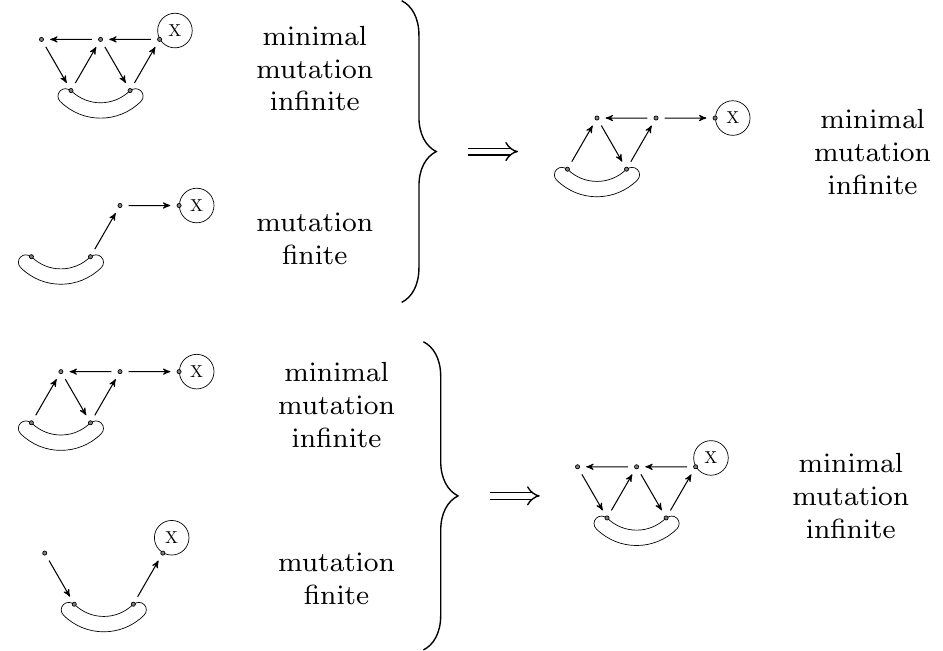}\quad
	\includegraphics{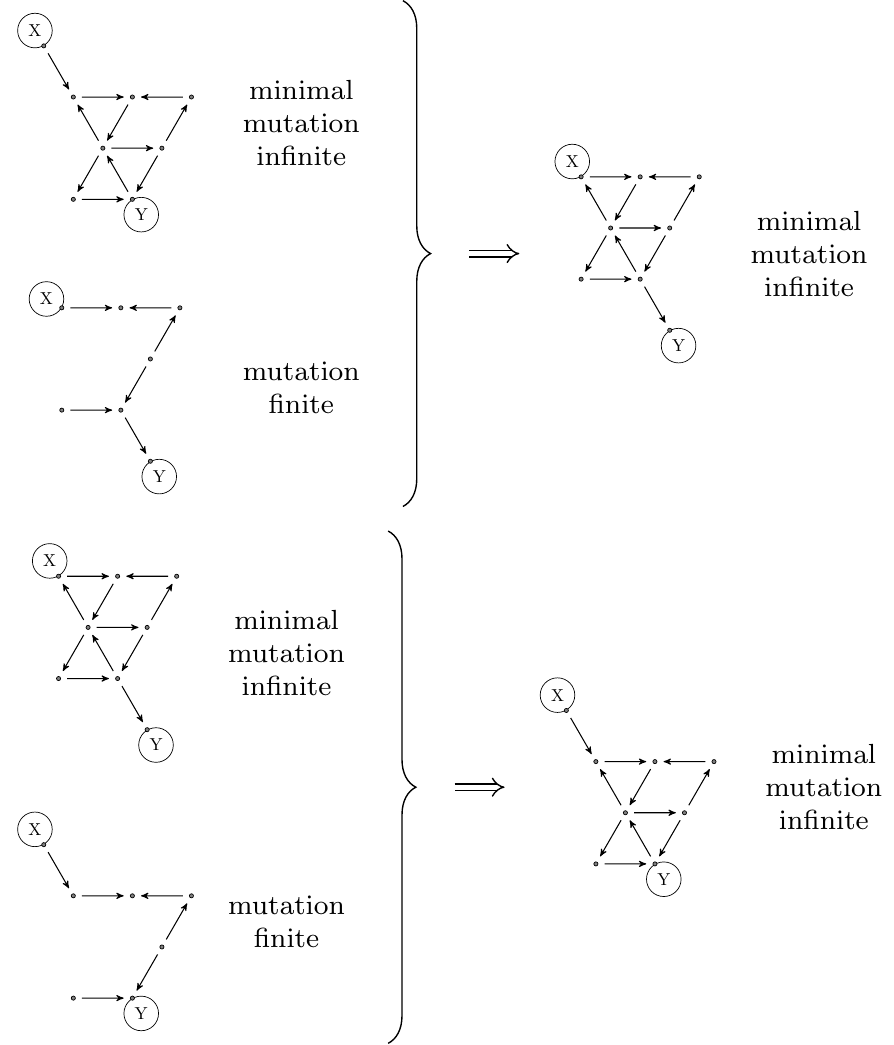}\quad
	\includegraphics{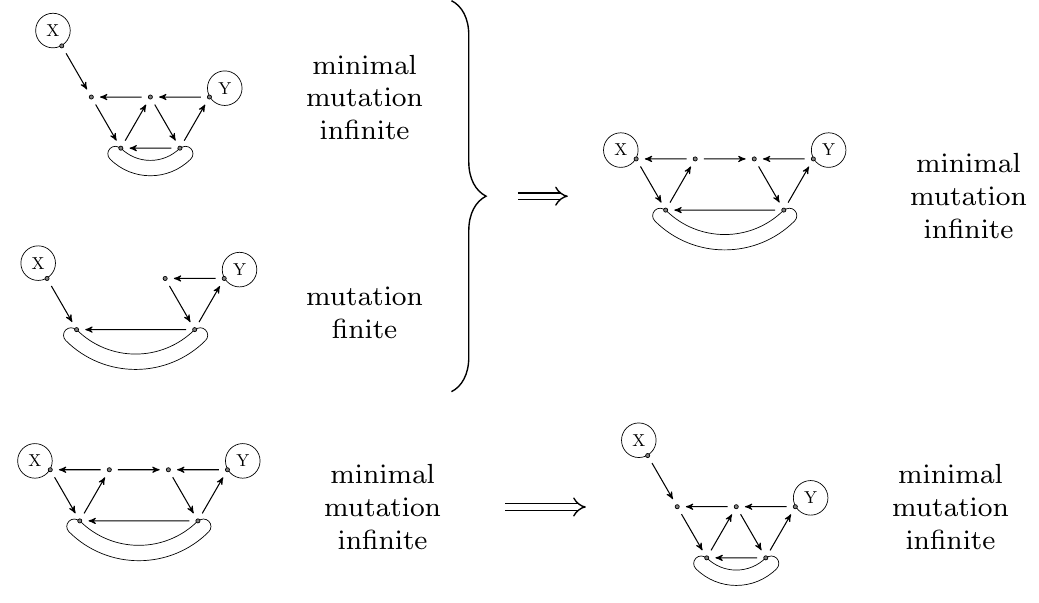}\quad
	\includegraphics{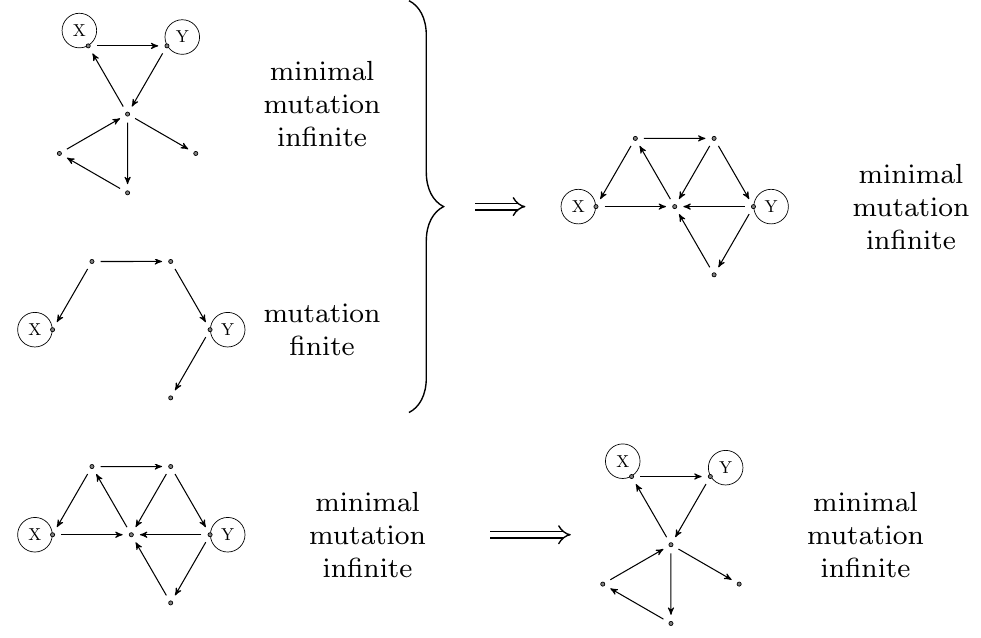}\quad
	\includegraphics{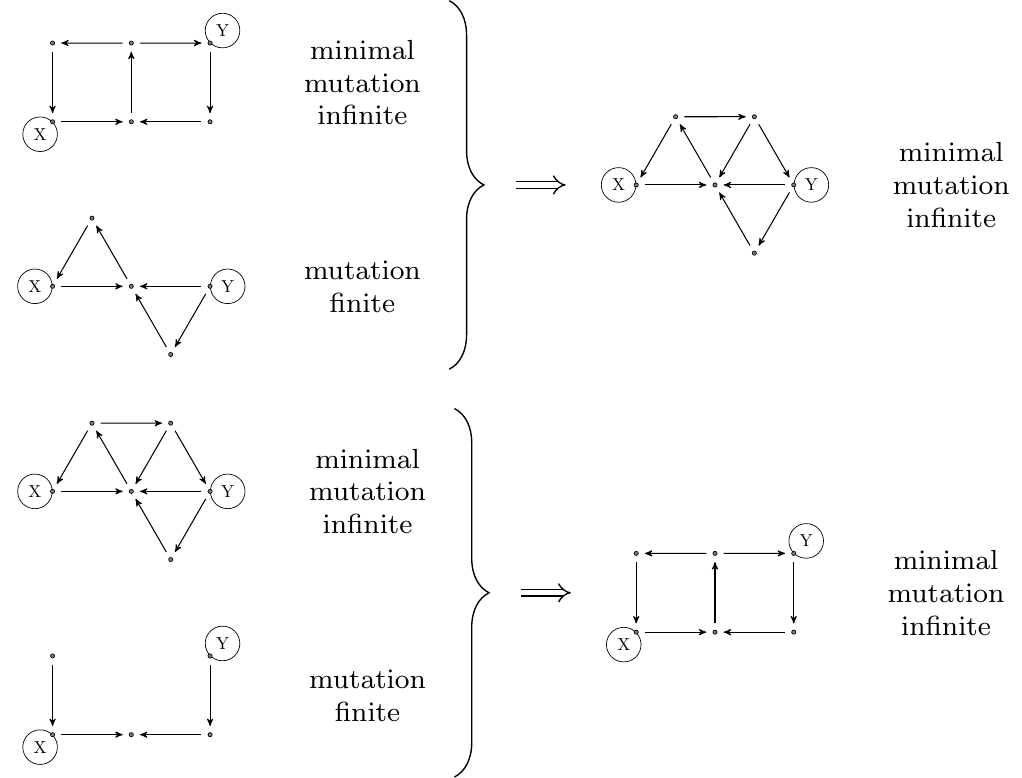}\quad
\end{nospacecenter}

\affiliationone{Department of Mathematical Sciences\\
	Durham University\\
	South Road\\
	Durham, UK\\
	DH1 3LE\\
\email{j.w.lawson@durham.ac.uk}}
\end{document}